\documentclass[a4paper,12pt]{article}
\usepackage[utf8]{inputenc}
\usepackage[T2A]{fontenc}
\usepackage{geometry}
\usepackage{amsmath}
\usepackage{amssymb}
\usepackage{graphicx}
\usepackage{setspace}
\usepackage{indentfirst}
\usepackage{hyperref}
\usepackage{arydshln}
\usepackage{amsthm}
\newsavebox{\Amat}

\usepackage{tikz}
\usetikzlibrary{matrix}

\usepackage{mathrsfs} 

\newtheorem{lemma}{Lemma}
\newtheorem{corollary}{Corollary}
\newtheorem{theorem}{Theorem}

\newtheorem{remark}{Remark}

\begin{document}
	\begin{center}
    \LARGE \textbf{On the exact discretization of the $L_2$-norm in spaces generated by Rademacher functions and their $p$-adic analogues}
\end{center}

\begin{center} \Large
Anna Kazakova \footnote{Lomonosov Moscow State University, anna.kazakova@math.msu.ru}
\end{center}


\begin{center} \textbf{Abstract}

Sufficient conditions for the exact discretization of the $L_2$-norm in the space spanned by generalized Rademacher functions are proved. It is shown that the sufficient number of nodes for discretization with the minimal number of nodes is equal to $N$ or $N+1$ and depends on the dimension $N$. The conjecture of F.~Dai, A.~Prymak, V.~N.~Temlyakov, and S.~Yu.~Tikhonov on the positivity of weights in the exact discretization of the $L_2$-norm with a minimal number of points considered. It was investigated in which cases it holds and in which it fails for the spaces spanned by the  $N$ generalized Rademacher functions with base $p$, and how the validity of the conjecture in this case depends on the parameters $p$ and $N$. The connection with generalized Hadamard matrices, the Hadamard's conjecture and Seberry's conjecture is demonstrated.
\end{center}
\vspace{0.5em}
The systematic study of discretization by values at points $L_q$-norms of  
functions from a given finite-dimensional subspace started in 2017 (see \cite{Temlyakov-2017} and \cite{Temlyakov-2018}); although the first results in this direction were obtained in the 1930s, tracing back to S. N. Bernstein, J. Marcinkiewicz, and A. Zygmund. Currently, this is an extensive and actively developing area of research that has deep connections with other important areas (see the reviews \cite{Kashin-2022}, \cite{Limonova-2024}, \cite{D-P-T-T}).

The paper is devoted to the problem of exact discretization of the $ L_2$-norm.  Let $\Omega \subset \mathbb{R}$ be a compact subset and $\mu$ a finite measure on $\Omega$. Let $X_N \subset L_2(\Omega, \mu)$ be an $N$-dimensional subspace. We say that $X_N$ admits an \emph{exact Marcinkiewicz-type discretization theorem with weights} with parameters $m$ and $2$ if there exist a set of points $\{\xi^j\}_{j=1}^m \subset \Omega$ and a set of weights $\{\lambda_j\}_{j=1}^m$ such that for every function $f \in X_N$ the following identity holds:
\begin{equation} \label{eq:disc}
\int_\Omega |f|^2 \, d\mu = \sum_{j=1}^m \lambda_j |f(\xi^j)|^2.
\end{equation}

For brevity, this property is denoted by $X_N \in \mathscr{M}^w(m, 2, 0)$. If in~\eqref{eq:disc} all the weights of $\lambda_j$ are non-negative, write $X_N \in \mathscr{M}_+^w(m, 2, 0)$.

According to the result from~\cite[theorem~3.1]{D-P-T-T}, the inclusion $X_N \in \mathscr{M}^w(N(N+1)/2, 2, 0)$ holds. We study the case of the minimal possible number of points, that is, the quantity
\begin{equation}
m(X_N) := \min\{m : X_N \in \mathscr{M}^w(m, 2, 0)\} \; \text{и} \; m_{+}(X_N) := \min\{\, m : X_N \in \mathscr{M}^w_+(m,2,0) \,\}.
\end{equation}

In the paper of F.~ Dai, A. ~ Primak, V.~N. ~ Temlyakov, S.~Yu.~Tikhonov~\cite{D-P-T-T} the following conjecture was proposed:

\textbf{Сonjecture A} [Open Problem 3] {\it Let $m = m(X_N)$. Let points $\{\xi^j\}_{j=1}^{m} \subset \Omega$ and weights $\{\lambda_j\}_{j=1}^{m}$ such that for any function $f \in X_N$ the equality~\eqref{eq:disc} holds, then all weights are positive: $\lambda_j > 0$ for $j = 1, \ldots, m$.}

In the paper~\cite{Limonova-2021} by I. V. Limonova \textbf{conjecture A} was refuted: the counterexamples are specially constructed subspaces for which the conjecture \textbf{A} does not hold. Namely, in Theorem 1 of \cite{Limonova-2021}, author constructs an example of a space $X_2$ in which, for the minimum number of $m$ points required to discretize the $L_2$–norm of functions from the subspace $X_2$, there exists a representation of the form ~\eqref{eq:disc} as with all positive weights both with one negative and two positive weights. And in Theorem 2 of \cite{Limonova-2021}, author constructs an example of a space $X_8$ in which, for the minimum number of $m$ points required to discretize the $L_2$–norm of functions from the subspace $X_8$, there is no representation of the form ~\eqref{eq:disc} with all positive weights, however, there is a representation with one negative and eight positive. Thus, Theorem 2 of \cite{Limonova-2021} refutes even the strengthened version of conjecture \textbf{A} from \cite{D-P-T-T}: 

\textbf{Сonjecture A$^\prime$} {\it Let $m = m(X_N)$ then among all variants of points and weights there are such points $\{\xi^j\}_{j=1}^{m} \subset \Omega$ and weights $\{\lambda_j\}_{j=1}^{m}$ that for any function $ f \in X_N $ the equality~\eqref{eq:disc} holds and all weights are positive.} 

Nevertheless, the natural question is about the validity of the \textbf{conjecture A} and \textbf{A$^\prime$} remained completely unexplored in the case of spaces generated by classical systems of functions. 

In this paper, we consider generalized Rademacher functions with base $p$ and the subspaces generated by them. As a consequence, we also obtain results for the classical Rademacher system.

Let $ \omega_p := \exp ( 2\pi i / p )$. Generalized Rademacher functions on $[0,1)$ are defined by the formula $R_k (x): = \omega_p^{ \lfloor x \cdot p^{k+1} \rfloor }$, 
$k \in  \mathbb{N}_0$. 
It is easy to show that $R_k (x) = \omega_p^{ x_k }$, 
where $x_k$ are coefficients from the expansion
\begin{equation} \label{expantion}
    x = \sum_{k=0}^{\infty} x_k p^{-k-1}, 
\qquad 
x_k = 0, \ldots,  p - 1,
\end{equation}  of the number $x$ into an infinite base-$p$ fraction. 
If there are two such decompositions, 
we take the one for which $x_k=0$, 
starting from some $k$. For $p=2$, the Rademacher function is denoted by $r_k$.

For a fixed set $ \mathbf{s}_N = \{s_1, \ldots, s_N\}$, where $s_i \in \mathbb{N}$, $N \ge 2$ and $0\le s_1 < s_2 < \ldots <s_N$, and an arbitrary positive integer $p \ge 2$ consider the space:
\[
\mathbf{R}_{\mathbf{s}_N}^p := \langle R_{s_1}, \ldots, R_{s_N} \rangle, \qquad \text{where $R_{k}$ are Rademacher functions with base $p$.} 
\]

Let's briefly describe the content of the paper. The section  \ref{BAZA}  contains auxiliary statements (Lemmas \ref{lemma-baza}--\ref{lemma-dop3}). The section \ref{MAINC} contains the main results (Theorems \ref{theorem-1}--\ref{theorem-3}). It is proved that the \textbf{conjecture A} for spaces $ \mathbf{R}_{\mathbf{s}_N}^p$ is true for all $N \le 4 \sin^2 (\pi/p)$ and for $N > 4 \sin^2(\pi/p)$ is true if there exists a generalized Hadamard matrix over $ \mathbb{Z}_p$ and not otherwise (see Theorems \ref{theorem-1}-\ref{theorem-2} and Corollary \ref{th-1-2}). It is proved that for $p=4$, the conjecture $ \mathbf{A^ \prime}$ for spaces $\mathbf{R}_{\mathbf{s}_N}^4$ is equivalent to Seberry's conjecture on the existence of generalized Hadamard matrices over $ \mathbb{Z}_4$ for all even $N$ (see Theorem \ref{theorem-5}). It is proved that for $p=2$ and $p=3$ there exist infinitely many natural numbers $n$  for which the conjecture $\mathbf{A^\prime}$ is not true (see Theorems \ref{theorem-4}-\ref{theorem-3}), also, the Theorem \ref{theorem-3} indicates a connection with the Hadamard hypothesis.

We call a generalized Hadamard matrix over $\mathbb{Z}_p$ of order $N$ the matrix $H_N$ whose elements belong to the set $\{1, \omega_p, \ldots, \omega_p^{p-1}\}$ and the identity $H_NH_N^{*} = NE_N$ holds, where $E_N$ is the identity matrix. 

Generalized Hadamard matrices over $\mathbb{Z}_p $ of order $N$ were first studied by Batson (see \cite{Butson}), so such matrices are also called Batson type Hadamard matrices $(p, N)$.
\section{Auxiliary statements} \label{BAZA}
\begin{lemma} \label{lemma-baza}
The following conditions are equivalent:
\begin{enumerate}
    \item $\mathbf{R}_{\mathbf{s}_N}^p \in \mathscr{M}^{\omega}(m, 2, 0)$,
 \item There are points $\{\xi^k\}_{k=1}^{m} \subset [0, 1)$ and weights $\{\lambda_k\}_{k=1}^{m} \subset \mathbb{R} \setminus \{0\}$ such that \begin{equation} \label{main}
\begin{cases}
\displaystyle \sum_{k=1}^m \lambda_k = 1,  \\[15pt]
\displaystyle \sum_{k=1}^m \lambda_k R_{s_i}(\xi^k) \overline{R_{s_j}(\xi^k) } = 0, & i < j, \quad i,j = 1, \dots, N.
\end{cases}
\end{equation}
\item There are points $\{\xi^k\}_{k=1}^{m} \subset [0, 1)$ and weights $\{\lambda_k\}_{k=1}^{m} \subset \mathbb{R} \setminus \{0\}$ such that
\begin{equation} \label{4}
A_{m} \Lambda A_{m}^{*} = E_N, 
\end{equation}
where $A_m = \displaystyle
\left(
\begin{array}{ccccc}
R_{s_1}(\xi^1) & \cdots &  R_{s_1}(\xi^m)  \\[5pt]
\vdots  & & \vdots \\[5pt]
R_{s_N}(\xi^1) & \cdots &
R_{s_N}(\xi^m) \\[5pt]
\end{array}
\right) \in \{1, \ldots, \omega_p^{p-1}\}^{N\times m}, \qquad \Lambda = \mathrm{diag}[\lambda_1, \ldots, \lambda_m].
$

We will write $A_m(\xi^1, \ldots, \xi^m)$ if we want to specify the dependence on points $\xi^1, \ldots, \xi^m$.
\end{enumerate}
\end{lemma}
\begin{proof}
 Let us show the equivalence of statements 1 and 2. If $ \mathbf{R}_{\mathbf{s}_N}^p \in \mathscr{M}^{\omega}(m, 2, 0)$, then there exists a set of points $\{\xi^k\}_{k=1}^{m} \subset [0,1)$ and a set of weights $\{\lambda_k\}_{k=1}^{m} \subset \mathbb{R}\setminus\{0\}$ such that for an arbitrary function $f \in \mathbf{R}_{\mathbf{s}_N}^p$ holds \eqref{eq:disc}. Any function $f \in \mathbf{R}_{\mathbf{s}_N}^p$ has the form $f = \sum_{i=1}^{N} a_i R_{s_i}$. Then for any $(a_1, \ldots, a_N) \in \mathbb{C}^N$ 
\begin{equation*}
\int_0^1 \biggl|\sum_{i=1}^{N} a_i R_{s_i}\biggl|^2 \, d x = \sum_{j=1}^m \lambda_j  \biggl|\sum_{i=1}^{N} a_i R_{s_i}(\xi^j) \biggl|^2.
\label{eq:discretization}
\end{equation*}
Which implies that
\[
\sum_{i=1}^{N}\sum_{j=1}^{N} a_i a_j \int_0^1 R_{s_i}\overline{R_{s_j}}dx = \sum_{i=1}^{N}\sum_{j=1}^{N} a_ia_j \sum_{k=1}^{m} \lambda_kR_{s_i}(\xi^k)\overline{R_{s_j}}(\xi^k).
\]
We equate the coefficients for the monomials $a_ia_j$. Then, taking into account the orthonormality of the Rademacher system and that the weights $ \lambda_k$ are real, we obtain the system \eqref{main} from $\frac{N(N-1)}{2} +1 $ equations.

That is, we got 1 $ \Rightarrow$ 2, doing the same reasoning in the opposite direction, we get 2 $\Rightarrow$ 1.

Let us show the equivalence of statements 2 and 3. Calculate the elements of the matrix on the right side of \eqref{4}:
\[
(A_m\Lambda A_m^{*})_{ij} = \sum_{k=1}^{m} R_{s_i}(\xi^k)\lambda_k \overline{R_{s_i}(\xi^k)}.
\]
Writing down the element-wise matrix identity \eqref{4}, we have:
\begin{equation} \label{4*}
    \sum_{k=1}^{m}R_{s_i}(\xi^k)\lambda_k\overline{R_{s_j}(\xi^k)} = \delta_{ij}
\end{equation}
Given that out of $N^2 $ equations in \eqref{4*} different $ \frac{N(N+1)}{2}+1$, we get that the systems \eqref{main} and \eqref{4*} (i.e., and \eqref{4}) are equivalent.
\end{proof}

\begin{remark}
Since the functions $R_{s_1},\dots,R_{s_N}$ are constant on the $p$-adic intervals
$\Delta_m^{(s_N+1)} = \bigl[\frac{m}{p^{s_N+1}}, \frac{m+1}{p^{s_N+1}}\bigl) $ of rank $ s_N+1$, the integral $L_2$-norm in $\mathbf{R}_{\mathbf{s}_N}^p$ is in fact
computed with respect to the discrete measure uniformly distributed on $p^{s_N+1}$
points, which corresponds to the solution
$\lambda_1=\dots=\lambda_{p^{s_N+1}}=\frac{1}{p^{s_N+1}}$
of the matrix system~\eqref{2} with unknowns
$\lambda_1,\dots,\lambda_{p^{s_N+1}}\in\mathbb{R}$:
\begin{equation} \label{2}
    A \Lambda A^{*} = E_N, \qquad \text{где}
\end{equation}
\begin{lrbox}{\Amat}
$\displaystyle
\left(
\begin{array}{ccccc}
R_{s_1}(\Delta_0^{(s_N+1)}) & \cdots &  R_{s_1}(\Delta_{p^{s_N+1}  - 1}^{(s_N+1)})  \\[5pt]
\vdots  & & \vdots \\[5pt]
R_{s_{N}}(\Delta_0^{(s_N+1)}) & \cdots &
R_{s_{N}}(\Delta_{p^{s_N+1}-1}^{(s_N+1)}) \\[5pt]
\end{array}
\right)
$
\end{lrbox}

\[
A :=
\underbrace{\usebox{\Amat}}_{p^{s_N+1}}
\qquad
\left.\vphantom{\usebox{\Amat}}\right\}
N, \quad  \Lambda = \mathrm{diag}[\lambda_1, \ldots, \lambda_{p^{s_N+1}}].
\]
The problem of minimizing the number of points sufficient to discretize $L_2$-norms in $\mathbf{R}_{\mathbf{s}_N}^{p}$ can be viewed as the search for a solution of the  system \eqref{2} in which as many as possible $\lambda_j$  are zero. Then if zero columns are removed from the matrix $ \Lambda$ and columns with the same numbers are removed from the matrix $A$, then the system for the submatrices obtained in this way will have the form \eqref{4} in which we assume that all $\lambda_j \neq 0$.

\end{remark}

\begin{lemma} \label{lemma-1}
 If $ \mathbf{R}_{\mathbf{s}_N}^p \in \mathscr{M}^{\omega}(m, 2, 0)$, then $m \ge N$.
\end{lemma}
\begin{proof} By Lemma \ref{lemma-baza}, there are points $\{\xi^k\}_{k=1}^{m} \subset [0, 1)$ and weights $\{\lambda_k\}_{k=1}^{m} \subset \mathbb{R} \setminus \{0\}$ such that \eqref{4} is executed. Then we get
\[
N
=
\operatorname{rank} E_N
=
\operatorname{rank}(A_{m} \Lambda A_{m} ^{T})
\le
\operatorname{rank} A_{m}
\le m.
\]
\end{proof}

\begin{lemma} \label{lemma-dop}
 If $ \mathbf{R}_{\mathbf{s}_N}^p \in \mathscr{M}_{+}^{\omega}(m, 2, 0)$, then the inequality holds for each weight:
    \[
      0 < \lambda_j \le \frac{1}{N}, \qquad j=1 \ldots, m.
    \]
\end{lemma}
\begin{proof}
 Since $ \mathbf{R}_{\mathbf{s}_N}^p \in \mathscr{M}_{+}^{\omega}(m, 2, 0)$, the matrix equality \eqref{4} holds. Consider the $j$th column of the matrix $A_m$:
    \[
    A_m^j = \left(
\begin{array}{ccccc}
R_{s_1}(\xi^j)  \\[5pt]
\vdots \\[5pt]
R_{s_N}(\xi^j)  \\[5pt]
\end{array}
\right).
    \]
Multiply the matrix equality \eqref{4} on the left by $(A_m^j)^{*}$ and on the right by $(A_m^j)$. Given that $(A_m^j)^{*} \cdot(A_m^j)=N$, we get:

\begin{equation} \label{lemma-dop-eq} 
(A_m^j)^{*}\cdot A_m^{*} \Lambda A_m \cdot (A_m^j) = N.
\end{equation}
Calculate and estimate the left one in \eqref{lemma-dop-eq}.
\begin{equation} \label{lemma-dop-eq1} 
(A_m^j)^{*}\cdot A_m^{*} \Lambda A_m \cdot (A_m^j) = \sum_{k=1}^m \lambda_j |(A_m^j)^{*}\cdot(A_m^k)|^2 > \lambda_j |(A_m^j)^{*}\cdot(A_m^j)|^2 = \lambda_jN^2, 
\end{equation}
where the estimation used the positivity of the weights $\lambda_j$.

Then from \eqref{lemma-dop-eq} and \eqref{lemma-dop-eq1} we get that
\begin{equation} \label{lemma-dop-eq3} 
\lambda_jN^2 < N.
\end{equation}
And \eqref{lemma-dop-eq3} immediately follows the statement of the lemma.
\end{proof}
\begin{lemma} \label{lemma-dop2}
Let there be points $\xi^1, \ldots, \xi^{N+1}$ and positive weights $\lambda_1, \ldots \lambda_{N+1}$ such that $\mathbf{R}^p_{\mathbf{s}_N} \in \mathscr{M}_+^{\omega}(N+1, 2, 0)$. Then there is a string vector $(u_1, \ldots, u_{N+1})$ of unit length such that the system holds:
\begin{equation} \label{eqv-syst0}
    \begin{cases}
1 = N \lambda_k + |u_k|^2, \quad k =1, \ldots, N+1 \\[6pt]
0 = \sqrt{\lambda_k \lambda_l} \sum_{i=1}^{N}R_{s_i}(\xi^k)\overline{R_{s_i}(\xi^l)} + u_k \overline{u_l}, \quad k \neq l.
\end{cases}
\end{equation}
\end{lemma}
\begin{proof}
 According to the Lemma \ref{lemma-baza}, the system \eqref{4} is executed. Since all weights are positive, there is $\sqrt {\Lambda}$ and the equality \eqref{4} can be written as: 
\[
(A_{N+1} \sqrt{\Lambda})(A_{N+1} \sqrt{\Lambda})^{*}=E_N.
\]

Define $B:=A_{N+1} \sqrt {\Lambda} \in \mathbb{C}^{(N+1)\times N}$
and add the matrix $B$ to a square matrix of order $N+1$ by adding a vector string $ \mathbf{u}$ of unit length and orthogonal to all rows of the matrix $B$. The resulting matrix $\widetilde{B} = \begin{pmatrix}
B \\
\mathbf{u}
\end{pmatrix} \in  \mathbb{C}^{(N+1)\times (N+1)} $ is an orthogonal matrix: 
\begin{equation} \label{Ort}
   E_{N+1} = \begin{pmatrix}
BB^{*} & 0\\
0 &  \mathbf{u}\mathbf{u}^{*}
\end{pmatrix} = \widetilde{B}\widetilde{B}^{*}.
\end{equation}
On the other hand, it executes: \begin{equation} \label{Ort2}
  \widetilde{B}\widetilde{B}^{*} = \widetilde{B}^{*} \widetilde{B} = B^*B + \mathbf{u}^*\mathbf{u}.
\end{equation}
Equating the left part in \eqref{Ort} and the right part in \eqref{Ort2}, we get the system \eqref{eqv-syst0}.
\end{proof}
\begin{lemma} \label{lemma-dop3} Let $p=2$ or $p=4$, and let $N$ be odd. Let there be points $\xi^1, \ldots, \xi^{N+1}$ and positive weights $\lambda_1, \ldots \lambda_{N+1}$ such that $\mathbf{R}^p_{\mathbf{s}_N} \in \mathscr{M}_+^{\omega}(N+1, 2, 0)$. Then the equality holds
\[
\lambda_1 = \ldots = \lambda_{N+1} = \frac{1}{N+1}. 
\]
\end{lemma}
\begin{proof} In the case of $p=2$ or $p=4$ and odd $N$, we have
\begin{equation} \label{ge1}
    \big|\sum_{i=1}^{N}R_{s_i}(\xi^k)\overline{R_{s_i}(\xi^l)}\big| \ge 1.
\end{equation} 
In the case of $p=2$, the inequality \eqref{ge1} is obvious. In the case of $p=4$, we have $ \big|\sum_{i=1}^{N}R_{s_i}(\xi^k)\overline{R_{s_i}(\xi^l)} \big| = a+ i b$ for some integers $a$ and $b$, with $a+b$ having the same evenness as $N$, i.e. odd. This means that at least one of the numbers $a$ or $b$ is not zero, i.e. $a^2 + b^2 \ge 1$. Then
\[
\big|\sum_{i=1}^{N}R_{s_i}(\xi^k)\overline{R_{s_i}(\xi^l)}\big| = \sqrt{a^2+b^2} \ge 1.
\]
By Lemma \ref{lemma-dop2}, the system \eqref{eqv-syst0} holds. Then it follows from \eqref{ge1} and the second equation in \eqref{eqv-syst0} that $|u_k|^2 |u_l|^2 \ge \lambda_k \lambda_l$. Then multiplying the equality $ \sum_{k=1}^{N+1} \lambda_k =1$ by $\lambda_l$, we get
\begin{equation}  \label{eqv-ineqv}
|u_l|^2 - |u_l|^4 + \lambda_l^2 = \sum_{k=1}^{N+1} |u_k|^2 |u_l|^2 - |u_l|^4 + \lambda_l^2 \ge \sum_{k=1}^{N+1} \lambda_k \lambda_l =  \lambda_l.
\end{equation}
Then from \eqref{eqv-ineqv} we get that
\begin{equation}  \label{eqv-ineqv-2}
(|u_l|^2-\lambda_l)(1-|u_l|^2-\lambda_l) \ge 0.
\end{equation}
Moreover, it is easy to make sure that $(1-|u_l|^2-\lambda_l) > 0$. Otherwise,$ |u_k|^2 + \lambda_k \ge 1 $ and substituting $|u_k|^2 = 1-N \lambda_k$ from \eqref{eqv-syst0} would get $\lambda_k \le 0$, which contradicts the assumption that the weights are positive. Then
\begin{equation}  \label{eqv-ineqv-3}
|u_l|^2 \ge \lambda_l.
\end{equation}
Substituting \eqref{eqv-ineqv-3} into the first equation in \eqref{eqv-syst0}, we get that $\lambda_k \le \frac{1- \lambda_k}{N}$, i.e.
\begin{equation} \label{Big}
   \lambda_k  \le \frac{1}{N+1} \qquad \forall k \in \{1, \ldots, N+1\}. 
\end{equation}
Given $ \sum_{k=1}^{N+1} \lambda_k =1$, we get that
\[
\lambda_k  =  \frac{1}{N+1},  \quad |u_k|^2 = \frac{1}{N+1}, \qquad \forall k \in \{1, \ldots, N+1\}.
\]
\end{proof}
\section{Main results} \label{MAINC}
\begin{theorem} \label{theorem-1} For every $N$, the inclusion $\mathbf{R}_{\mathbf{s}_N}^p \in \mathscr{M}^{\omega}_{+}(M,2,0)$ holds, where $M \ge N$ and is such that there exists a generalized Hadamard matrix over $\mathbb{Z}_p$ of order $M$.

The condition $m( \mathbf{R}_{\mathbf{s}_N}^p) = N$ is equivalent to the existence of a generalized Hadamard matrix over $\mathbb{Z}_p$ of order $N$. In this case, the points $\xi^1,\ldots, \xi^N$ are such that $A_{N}(\xi^1, \ldots \xi^N)$ is a Hadamard matrix, and the weights are $\lambda_1 = \ldots = \lambda_N = \frac{1}{N}$.
\end{theorem}
\begin{proof}
 First, we prove the second part of the theorem. If the number of points $m =N$, then by Lemma \ref{lemma-1} this is the minimum number of points, i.e. $m(\mathbf{R}_{\mathbf{s}_N}^p)$. Then for a square matrix $A_{N}$ from \eqref{4}, we get that $|\det A_N|^2 = (\det \Lambda)^{-1}$. Since all $ \lambda_j \neq 0$ in \eqref{4}, $A_N$ is nonsingular. Multiplying the identity \eqref{4} on the left by $A_N^{-1}$ and on the right by $(A^{*}_N)^{-1}$, we get
\[
\Lambda=A_{N}^{-1}(A_{N}^{*})^{-1}.
\]
Therefore,
\begin{equation} \label{ALambda}
A_{N}^{*}A_{N} = \Lambda^{-1}.    
\end{equation}
Since the values of the generalized Rademacher functions lie on the unit circle, the diagonal elements of the matrix $A_{N}^{*}A_{N}$ are equal to $N$. Then from \eqref{ALambda} we get $\lambda_1=\ldots=\lambda_N = \frac{1}{N}$ and
\[
A_{N}^{*}A_{N} = NE_N,
\]
that is, $A_{N}$ is a generalized Hadamard matrix over $ \mathbb{Z}_p$.

Conversely, let $N$ be such that a generalized Hadamard matrix $H_N$ over $\mathbb{Z}_p$ of order $N$ exists. Taking it as a matrix $A_{N}$ and taking equal weights $ \frac{1}{N}$, we get the system \eqref{4} for them, i.e. $\mathbf{R}_{\mathbf{s}_N}^p \in \mathscr{M}^{\omega}(N, 2, 0)$ for such $N$.

Let us prove the first part of the theorem. We show that if there is no generalized Hadamard matrix over $\mathbb{Z}_p$ of order $N$, then taking the Hadamard matrix $\mathbb{Z}_p$ of order $M$ greater than $N$, we obtain a representation of the form \eqref{eq:disc} c $M$ points. Let $H_M$ be a generalized Hadamard matrix over $ \mathbb{Z}_p$ of order $M$. We get the matrix $A_M$ from it by taking any $N$ rows of the matrix $H_M$. Then $A_{M}A_{M}^* = ME_N$, from which we immediately get a system of the form \eqref{4} for the matrix $\Lambda$ with $\lambda_1 = \ldots = \lambda_M=\frac{1}{M}$. The latter by Lemma \ref{lemma-baza} is equivalent to including $\mathbf{R}_{\mathbf{s}_N}^p \in \mathscr{M}^{\omega}_{+}(M, 2, 0)$.
\end{proof}

Let us show that if we move slightly away from the minimal number $N$ and consider an exact discretization of the $L_2$-norm with $(N+1)$ points in the space $\mathbf{R}_{\mathbf{s}_N}^p$, then a construction naturally arises in which $N$ of these points are assigned positive weights and one of these points is assigned a negative weight.

\begin{theorem} \label{theorem-2} 
For all $N \ge 2$ and $p \ge 2$, the following inclusion holds:
\begin{equation} \label{include}
\mathbf{R}_{\mathbf{s}_N}^p \in \mathscr{M}^{\omega}(N+1, 2, 0).
\end{equation}
If $N > 4\sin^2(\pi/p)$, then one of the $N+1$ weights can be made negative and the remaining $N$ positive.
\end{theorem}
\begin{proof}
We will look for a solution with $N$ weights equal to each other and equal to a certain number of $\alpha$ and $(N+1)$- th weight is equal to some number $\beta$. By virtue of the second equality in \eqref{main}: 
\begin{equation} \label{beta-alpha}
    \beta = 1-N\alpha.
\end{equation}

According to Lemma \ref{lemma-baza}, the condition \eqref{include} can be written in the matrix form \eqref{4}. In our case, the matrix \begin{equation} \label{1055}
    \Lambda := \mathrm{diag}[\alpha, \alpha, \ldots, \alpha, \beta],
\end{equation} and the matrix $A_{N+1}$ will be searched in the form
\begin{equation} \label{1056}
A_{N+1} = \left(
\vcenter{\hbox{
\begin{tikzpicture}
    \matrix (m) [
        matrix of math nodes,
 nodes={font=\Large, inner sep=1mm}, 
 column sep=1mm 
    ] {
        |[inner xsep=8mm, inner ysep=8mm]| B & 
        |[inner xsep=2mm, inner ysep=8mm]| \vphantom{B} \textbf{c} \\
    };

    \draw[dashed, semithick] (m-1-1.north west) rectangle (m-1-1.south east);
    \draw[dashed, semithick] (m-1-2.north west) rectangle (m-1-2.south east);
    
    \draw[semithick] (m-1-1.east -| m-1-1.north) -- (m-1-1.east -| m-1-1.south);
\end{tikzpicture}
}}
\right), \qquad \text{где $B \in \{1, \ldots, \omega_p^{p-1}\}^{N\times N}$, $\mathbf{c} \in \{1, \ldots, \omega_p^{p-1}\}^{1\times N}$}.
\end{equation}
Then the equality \eqref{4} for the matrices \eqref{1055} and \eqref{1056} has the form
\begin{equation}
    \alpha B B^{*} + \beta \mathbf{c}\mathbf{c}^{*} = E_N.
\end{equation}
Let's put the vector $ \mathbf{c} = (1, \ldots, 1)$. Then $ \mathbf{c}\mathbf{c}^{*}$ is a matrix $N\times N$, all elements of which are equal to 1. We get
\[E_N - \beta  \mathbf{c}\mathbf{c}^{*} = \begin{pmatrix}
1-\beta & -\beta & \cdots & -\beta \\
-\beta & 1-\beta & \cdots & -\beta \\
\vdots & \vdots & \ddots & \vdots \\
-\beta & -\beta & \cdots & 1-\beta
\end{pmatrix} \in \mathbb{R}^{N\times N}.  \] 
Given that $ \alpha = \frac{1-\beta}{N}$, we get the condition for the matrix $B$:
\begin{equation} \label{1122}
    \frac{1}{N} BB^{*} = \begin{pmatrix}
1 & -\frac{\beta}{1-\beta} & \cdots & -\frac{\beta}{1-\beta} \\
-\frac{\beta}{1-\beta} & 1 & \cdots & -\frac{\beta}{1-\beta} \\
\vdots & \vdots & \ddots & \vdots \\
-\frac{\beta}{1-\beta} & -\frac{\beta}{1-\beta} & \cdots & 1
\end{pmatrix}.
\end{equation}
That is, the matrix $B$ is such that the matrix $\frac{1}{N} BB^{*}$ has diagonal elements equal to $1$ and off-diagonal elements are equal to each other and equal to some number not equal to $1$. Let's put
\begin{equation} \label{1155}
    B = \begin{pmatrix}
\omega_p & 1 & \cdots & 1 \\
1 & \omega_p & \cdots & 1 \\
\vdots & \vdots & \ddots & \vdots \\
1 & 1 & \cdots & \omega_p
\end{pmatrix}.
\end{equation}
Then
\begin{equation} \label{1156}
    \frac{1}{N} BB^{*} = \begin{pmatrix}
1 & \frac{N - \rho_p}{N} & \cdots & \frac{N - \rho_p}{N} \\
\frac{N - \rho_p}{N} & 1 & \cdots & \frac{N - \rho_p}{N} \\
\vdots & \vdots & \ddots & \vdots \\
\frac{N - \rho_p}{N} & \frac{N - \rho_p}{N} & \cdots & 1
\end{pmatrix},
\end{equation}
\begin{equation} \label{rho_p}
\text{где} \quad \rho_p = 2- \omega_p - \overline{\omega}_p = |1 - \omega_p|^2 >0.
\end{equation}
Comparing the formulas \eqref{1122} and \eqref{1156}, we obtain the equation for weight $ \beta$:
\[
-\frac{\beta}{1-\beta} = \frac{N - \rho_p}{N}.
\]
Then $\beta = 1 - \frac{N}{\rho_p}$. So, given \eqref{beta-alpha}, $\alpha = \frac{1}{\rho_p}$.

The chosen matrix
\begin{equation} \label{choice}
    A_{N+1} = \begin{pmatrix}
\omega_p & 1 & \cdots & 1 & 1 \\
1 & \omega_p & \cdots & 1 & 1\\
\vdots & \vdots & \ddots & \vdots & \vdots \\
1 & 1 & \cdots & \omega_p & 1
\end{pmatrix}
\end{equation}
corresponds to the points
\begin{equation}     \label{F-k}
\begin{cases}
   & \xi^k \in \{x\in [0,1)\colon x_{s_k} = 1, x_{s_j} = 0 \ \text{при} \ i\in \{1, \ldots, N\}\setminus{k}\} =: F_k, \quad k = 1, \ldots N;\\
   & \xi^{N+1} \in \{x\in [0,1)\colon  x_{s_j} = 0 \ \text{при} \ i\in \{1, \ldots, N\}\} =: F_{N+1},
\end{cases} 
\end{equation}
where $x_k$ are coefficients from the expansion of the number $x$ into an infinite base-$p$ fraction (see \eqref{expantion}). The sets $F_k$ defined in \eqref{F-k} are arranged as unions of $p$-adic intervals of rank $s_N+1$, and $\mu(F_k) = \frac{1}{p^N}$ for $k = 1, \ldots N+1$:
\begin{equation} \label{F-k-Delta}
    \begin{cases}
   & F_{N+1} = \bigcup_{m} \Delta_m^{(s_N+1)}, \ \text{where the union by $m$ is} \  \bigl\{ \sum_{i=0}^{s_N}x_ip^{s_N-i} \colon x_{s_1}=\ldots =x_{s_N}=0\bigl\},\\
& F_k = F_{N+1} + \frac{1}{p^{s_k}}, \ \text{i.e. adding $\frac{1}{p^{s_k}}$ to each element of $x \in F_{N+1}$,} \ k = 1,\ldots, N.
\end{cases}
\end{equation}

Thus, we constructed a set of $N+1$ points $\{\xi^k\}_{k=1}^{N+1}$ (see \eqref{F-k}-\eqref{F-k-Delta}) and $N+1$ weights $\lambda_1 = \ldots =\lambda_{N} = \frac{1}{\rho_p}$ and $\lambda_{N+1} = 1 - \frac{N}{\rho_p}$ ($\rho_p$ defined in \eqref{rho_p}) for discretization $L_2$-norms, i.e. $\mathbf{R}_{\mathbf{s}_N}^p \in \mathscr{M}^{\omega}(N+1, 2, 0)$. Since $\rho_p>0$, so does $\lambda_i > 0$ for $i = 1, \ldots, N$. And $ \lambda_{N+1} < 0$ when $N >\rho_p$. Let's reduce the value of $ \rho_p$ to the form specified in the condition of the theorem:
\[
\rho_p = 2 - \omega_p - \overline{\omega}_p = 2(1 - \operatorname{Re}\omega_p) = 2\left(1 - \cos\frac{2\pi}{p}\right) = 4\sin^2\frac{\pi}{p}.
\]
\end{proof}

\begin{remark} 
Note that, according to formulas \eqref{F-k}–\eqref{F-k-Delta}, each point $\xi^k$ can be chosen from one of the $p^{s_N+1-N}$ $p$-adic intervals. That is, there are $p^{(s_N+1-N)(N+1)}$ suitable submatrices in the matrix $A$ defined in \eqref{2}. If, however, $\mathbf{s}_N = \{0, \ldots, N-1\}$, then in \eqref{F-k-Delta} there is one $p$-adic interval, and the chosen submatrix \eqref{choice} of the matrix $A$ from \eqref{2} is unique.
 \end{remark}

\begin{remark} \label{hard}
When $m(\mathbf{R}_{\mathbf{s}_N}^p) = N$, a representation of the form \eqref{eq:disc} according to Theorem \ref{theorem-1} is possible only for a unique choice of weights $\lambda_1 = \ldots = \lambda_{N} = \frac{1}{N}$. Nevertheless, the sets (which are unions of $p$-adic intervals) from which the points $\xi^1, \ldots, \xi^{N}$ are chosen are defined non-uniquely, since Hadamard matrices are not unique for fixed $N$ and $p$. When $m(\mathbf{R}_{\mathbf{s}_N}^p) = N+1$, a representation of the form \eqref{eq:disc} is non-unique in the sense that both with the same weights $\lambda_1, \ldots, \lambda_{N+1}$ one can find different suitable sets $F_1, \ldots, F_{N+1}$ and $F_1^\prime, \ldots, F_{N+1}^\prime$ from which the points $\xi^1, \ldots, \xi^{N+1}$ are chosen, and there also exist representations of the form \eqref{eq:disc} in which both the sets of weights $\lambda_1, \ldots, \lambda_{N+1}$ and $\lambda_1^{\prime}, \ldots, \lambda_{N+1}^{\prime}$ are different, and the sets $F_1, \ldots, F_{N+1}$ and $F_1^{\prime}, \ldots, F_{N+1}^{\prime}$ are different. To confirm the above, we will give two examples. 

\underline{Example 1.} Let $p=2$ and $N \ge 2$. Let $ \mathbf{c}= (1, -1, \ldots, -1)$ in \eqref{1056} and take $ \lambda_{N+1}=\beta = 1 - \frac{N}{\rho_2} = \frac{4-N}{4}$. Then
\[
\frac{1}{N}BB^{T} = \begin{pmatrix}
1 & \frac{4-N}{N} &  \frac{4-N}{N} & \cdots & \frac{4 - N}{N} \\
\frac{4 - N}{N} & 1 & \frac{N - 4}{N} &  \cdots & \frac{N - 4}{N} \\
\vdots & \vdots & \vdots & \ddots & \vdots \\
\frac{4 - N}{N}  & \frac{N - 4}{N} & \frac{N - 4}{N} &  \cdots & 1
\end{pmatrix}.
\]
and one can take $\lambda_{1}=\ldots=\lambda_{N}=\frac{1}{\rho_2} = \frac{1}{4}$ and the matrix\[
A_{N+1} = \begin{pmatrix}
1 & 1 & 1 & \cdots & 1 & 1 & 1 \\
1 & 1&  -1 &\cdots & -1  & -1 & -1\\
1 & -1&  1 &\cdots & -1  & -1 & -1\\
\vdots & \vdots & \vdots & \ddots & \vdots & \vdots & \vdots \\
1 & -1 &  -1 & \cdots & -1 & 1 & -1
\end{pmatrix}.
\]
Then the corresponding points $\xi^1, \ldots, \xi^{N+1}$ are selected from the sets
\begin{align*}
&\xi^1 \in \{x\in [0, 1)\colon x_{s_i} = 0 \ \text{for} \  i \in \{1, \ldots, N\}\} =:F_1^{\prime}, \\
&\xi^k \in \{x\in [0,1)\colon x_{s_1}=0, x_{s_k} =0, x_{s_i} =1 \ \text{for} \ i \in \{2, \ldots, N\}\setminus\{k\}\}=:F_k^{\prime},\ k =2, \ldots, N, \\
& \xi^{N+1} \in \{x\in [0,1)\colon x_{s_1}=0, x_{s_i} =1 \ \text{for} \ i \in \{2, \ldots, N\}\}=:F_{N+1}^{\prime}.
\end{align*}
\underline{Example 2.} Let $p=5$ and $N=2$. Consider the matrix
\[
A_{3} = \begin{pmatrix}
\omega_5^2 & 1 & 1 \\
1 & \omega_5^2 & 1\\
\end{pmatrix}
\]
instead of a matrix $A_{3}$ of the form \eqref{choice}. Then $\frac{\beta}{1-\beta} = -\frac{\omega_5^2 +\omega_5^3}{2} = -\cos(\frac{4\pi}{5})= \frac{1+\sqrt{5}}{4}$, i.e. $\lambda_3 = \beta = \frac{1}{\sqrt{5}}$, а $\lambda_1=\lambda_2 = \frac{1}{2}(1 - \frac{1}{\sqrt{5}})$. The points $\xi^1, \xi^2, \xi^{3}$ are selected from the sets: 
\begin{align*}
&\xi^1 \in \{x\in [0,1)\colon x_{s_1}=2, x_{s_2} =0\}=:F_{1}^{\prime}, \quad  \xi^2 \in \{x\in [0,1)\colon x_{s_1}=0, x_{s_2} =2\}=:F_{2}^{\prime},\\
& \xi^{3} \in \{x\in [0,1)\colon x_{s_1}=0, x_{s_2} =0\}=:F_{3}^{\prime}.
\end{align*}
\end{remark}

\begin{remark} Note that $4\sin^2\frac{\pi}{p} \rightarrow 0$ for $p \rightarrow \infty$. For $p=2$, we have the Theorem \ref{theorem-2} constraint for having one negative weight $N \ge 5$, for $p=3$ we have $N \ge 4$, for $p=4$ we have $N \ge 3$, and for $p \ge 5$ has $N\ge2$. That is, for $p \ge 5$, a construction with $N$ positive weights and one negative weight is constructed already for all $N$ considered.
\end{remark}

\begin{remark} Note that the condition $N > 4 \sin^2 (\pi/p)$ in the Theorem \ref{theorem-2} is not restrictive, since for $N \le 4 \sin^2 (\pi/p)$, with an exact discretization of $L_2$-norm  with $m(\mathbf{R}_{\mathbf{s}_N}^p)$ points, all weights are positive. Indeed, for $p=2$ and $N=2$, $p=2$ and $N=4$, $p=3$ and $N=3$, $p=4$ and $N=2$, there exists a generalized Hadamard matrix over $ \mathbb{Z}_p$, and means by the Theorem \ref{theorem-1} $m(\mathbf{R}_{\mathbf{s}_N}^p) = N$, and the weights are all positive and equal to each other. The remaining cases are $p=2$ and $N=3$, $p=3$ and $N=2$. For these cases, according to Theorem \ref{theorem-1} and Theorem \ref{theorem-2} $m(\mathbf{R}_{\mathbf{s}_N}^p) = N+1$ and in Theorem  \ref{theorem-2} we construct a representation with equal positive weights and equal $ \frac{1}{N+1}$ for these cases. Since according to the Remark \ref{hard} solution with $(N+1)$-th point is not unique, then we need to show that a solution with negative weight does not exist. We will proceed similarly to the proof of \ref{theorem-4} below. Let's start with the case of $p=3$ and $N=2$. If $ \mathbf{R}_{\mathbf{s}_2}^3 \in \mathscr{M}^{\omega}(3, 2, 0)$, then there are points $\xi^1, \xi^2, \xi^3$ and weights $\lambda_1, \lambda_2, \lambda_3$ for which the system \eqref{main}  is executed using Lemma \ref{lemma-baza}:
\begin{equation*}
\begin{cases}
\lambda_1 + \lambda_2 + \lambda_3  = 1,  \\
\lambda_1 R_{s_1}(\xi^1) \overline{R_{s_2}(\xi^1) } +\lambda_2 R_{s_1}(\xi^2) \overline{R_{s_2}(\xi^2) } +\lambda_3 R_{s_1}(\xi^3) \overline{R_{s_2}(\xi^3) }  = 0.
\end{cases}
\end{equation*}
Grouping by values $R_{s_1}(\xi^k) \overline{R_{s_2}(\xi^k)}$:
\begin{equation} \label{easy-2}
\sigma_i = \sum_{k} \lambda_k, \ \text{where summation by } k\colon R_{s_1}(\xi^k) \overline{R_{s_2}(\xi^k)} = \omega_3^{i}, \ i=0, 1, 2.
\end{equation}
Then
\begin{equation} \label{easy-3}
\begin{cases}
    \sigma_0 + \sigma_1 + \sigma_2 = 1  \\
    \sigma_0 + \sigma_1\omega_3  + \sigma_2\omega_3^2 =0
\end{cases}
\end{equation}
From the second equation in \eqref{easy-3}, by virtue of the uniqueness of the minimal polynomial $ \omega_3$, we obtain that $\sigma_0 = \sigma_1 = \sigma_2$, and taking into account the first equation in \eqref{easy-3}, they are equal to $\frac{1}{3}$. This means that $\lambda_1 = \lambda_2 = \lambda_3 = \frac{1}{3}$ due to  \eqref{easy-2}. 

The second case is $p = 2$ and $N=3$. If $ \mathbf{R}_{\mathbf{s}_3}^2 \in \mathscr{M}^{\omega}(4, 2, 0)$, then there are points $\xi^1, \xi^2, \xi^3, \xi^4$ and weights $\lambda_1, \lambda_2, \lambda_3, \lambda_4$ and by Lemma \ref{lemma-baza} the system \eqref{main} of 4 equations is performed for them. Group by values $r_{s_1}(\xi^k)r_{s_2}(\xi^k)$ and $r_{s_2}(\xi^k) r_{s_3}(\xi^k)$:
\begin{align} 
 \sigma_{i2^0+j2^1} = \sum_{k} \lambda_k, \ & \text{where summation by } k\colon \label{easy-4}\\
    & r_{s_1}(\xi^k) r_{s_2}(\xi^k) = (-1)^i, \  r_{s_2}(\xi^k) r_{s_3}(\xi^k) = (-1)^j, \ i=0, 1, \ j=0, 1. \nonumber
\end{align}
Then
\begin{equation} \label{easy-5}
    \begin{bmatrix}
        1 & 1 & 1 & 1 \\
        1 & -1 & 1 & -1 \\
        1 & 1 & -1 & -1\\
         1 & -1 & -1 & 1
    \end{bmatrix}
    \begin{bmatrix}
        \sigma_0 \\
        \sigma_1 \\
        \sigma_2 \\
        \sigma_3
    \end{bmatrix}
    =
    \begin{bmatrix}
        1 \\
        0 \\
        0 \\
        0
    \end{bmatrix}
\end{equation}
The matrix on the right is a Hadamard matrix, which means that it is invertible and the solution of the \eqref{easy-5} system is $\sigma_1 = \sigma_2 = \sigma_3 = \sigma_4 = \frac{1}{4}$. This means that $\lambda_1 = \lambda_2 = \lambda_3 = \frac{1}{4}$ due to \eqref{easy-4}. 
\end{remark}

In the Remark \ref{about-matrix} , we present some well-known necessary conditions for the existence of generalized Hadamard matrices over $ \mathbb{Z}_p$. For more information about these and other conditions, see \cite[Сhapter ~1]{Ferenc}.

\begin{remark} \label{about-matrix} 
\begin{enumerate}
 \item If $p=p_1^{q_1}\cdot \ldots \cdot p_r^{q_r}$, where $p_1, \ldots, p_r$ are prime numbers and a generalized Hadamard matrix over $\mathbb{Z}_p$ exists, then $N = 1$ or $N = p_1n_1
 + \ldots+p_rn_r >1 $, where $n_i\in \mathbb{N}_{0}$,
 \item If $p$ is prime and a generalized Hadamard matrix over $\mathbb{Z}_p$ exists, then $N = 1$ or $N \equiv 0 \ \pmod p$,
 \item If a classical Hadamard matrix of order $N$ exists, then $N = 1, 2$ or $N \equiv 0 \ \pmod4$.
\end{enumerate}
Note that the necessary condition 2 is a special case of 1. And condition 3 is a strengthening of condition 2 for $p=2$.
\end{remark}
\begin{corollary} \label{cordop} Let $p=p_1^{q_1}\cdot \ldots \cdot p_r^{q_r}$, where $p_1, \ldots, p_r$ are prime numbers, and $N \ge 2$. Let $N \neq p_1n_1+ \ ldots+p_rn_r >1 $ for no $n_i \in \mathbb{N}_{0}$. Then $m (\mathbf{R}_{\mathbf{s}_N}^p)=N+1$.
\end{corollary}
In the case of $p=2$, the Corollary \ref{cordop} can be strengthened.
\begin{corollary} \label{cordop2} Let $ p=2$ and $N\ge 2$. Let $N \neq 2$ and $N \neq 4k$ for no natural $k$. Then $m (\mathbf{R}_{\mathbf{s}_N}^p)=N+1$.
\end{corollary}
\begin{corollary} \label{th-1-2}
\[m(\mathbf{R}_{\mathbf{s}_N}^p)=
\begin{cases}
N,
&
\text{if there exists a generalized Hadamard matrix of order $N$ over $\mathbb{Z}_p$},
\\[1ex]
N+1,
&
\text{otherwise}.
\end{cases}
\]
In the case of $m(\mathbf{R}_{\mathbf{s}_N}^p) = N$ all weights are positive and equal to each other. In the case of $m(\mathbf{R}_{\mathbf{s}_N}^p)= N+1$ and $N> 4 \ sin^2 (\pi/p) $ one can give an example with one weight negative, and the other $ N $ weights positive and equal to each other. 
\end{corollary}

The Corollary \ref{th-1-2} gives, in a certain sense, a complete answer to the question: {\it for which $p$ and $N$ conjecture  \textbf{A} for spaces $ \mathbf{R}_{\mathbf{s}_N}^p$ is correct?}. Сonjecture \textbf{A} for spaces $ \mathbf{R}_{\mathbf{s}_N}^p$ is true for all $N \le 4 \sin^2 (\pi/p)$ and for $N > 4 \sin^2(\pi/p)$ is true if there exists a generalized Hadamard matrix over $ \mathbb{Z}_p$ and is not true otherwise. 

 In Theorems \ref{theorem-5}-\ref{theorem-3} , we give a partial answer to the question: {\it for which $p$ and $N$ conjecture  $\mathbf{A^{\prime}}$ for spaces $\mathbf{R}_{\mathbf{s}_N}^p$ is correct?} That is, for which $N$ and $p$ the strict inequality $m_+(\mathbf{R}_{\mathbf{s}_N}^p) > m(\mathbf{R}_{\mathbf{s}_N}^p)$ holds?
 
As shown in Theorem \ref{theorem-2}, for an exact discretization of the $L_2$-norm with the minimal number of nodes in the space $\mathbf{R}_{\mathbf{s}_N}^p$, one can give an example with one negative weight; however, it does not follow in any way from the theorems above that there is no example with all positive weights for the minimal number of nodes.

\begin{remark} \label{about-matrix-2} 
We present two well-known hypotheses about the existence of generalized Hadamard matrices over $ \mathbb{Z}_2$ and over $\mathbb{Z}_4$.
\begin{enumerate}
 \item Hadamard's conjecture. A Hadamard matrix exists for all $N$ multiples of 4.
 \item Seberry's conjecture. A generalized Hadamard matrix over $\mathbb{Z}_4$ exists for all $N$ multiples of 2.
    \end{enumerate}
Note that if Seberry's conjecture is true, then Hadamard's conjecture is also true.
\end{remark}

\begin{theorem} \label{theorem-5}
Seberry's conjecture (see the remark \ref{about-matrix-2}) is equivalent to the  conjecture \textbf{A$^\prime$} for subspaces $\mathbf{R}_{\mathbf{s}_N}^4$, i.e. that $m_{+}(\mathbf{R}_{\mathbf{s}_N}^4) = m(\mathbf{R}_{\mathbf{s}_N}^4)$.
\end{theorem} 
\begin{proof}
 First, we prove some analog of the Theorem \ref{theorem-1}: {\it if $N$ is odd, then the equality $m_{+}(\mathbf{R}_{\mathbf{s}_N}^4) = N+1$ is equivalent to the existence of a generalized Hadamard matrix over $\mathbb{Z}_4$ of order $N+1$}. Since $N$ is odd, there is no generalized Hadamard matrix over $\mathbb{Z}_4$. If there exists a generalized Hadamard matrix over $\mathbb{Z}_4 $ of order $N+1$, then by Theorem \ref{theorem-1} $m_{+}(\mathbf{R}_{\mathbf{s}_N}^4) = N+1$. Conversely, if $m_{+}(\mathbf{R}_{\mathbf{s}_N}^4) = N+1$, then the system \eqref{eqv-syst0} is executed by Lemma \ref{lemma-dop2}:
    \begin{equation*}
    \begin{cases}
1 = N \lambda_k + |u_k|^2, \quad k =1, \ldots, N+1 \\[6pt]
0 = \sqrt{\lambda_k \lambda_l} \sum_{i=1}^{N}R_{s_i}(\xi^k)\overline{R_{s_i}(\xi^l)} + u_k \overline{u_l}, \quad k \neq l,
\end{cases}
\end{equation*}
 where $(u_1, \ldots, u_{N+1})$ vector that complements the matrix $A_{N+1}\sqrt {\Lambda}$ to a square orthogonal matrix. By Lemma \ref{lemma-dop3} $\lambda_k = \frac{1}{N+1}$ and $ |u_k| = \frac{1}{\sqrt{N+1}}$ for $k = 1, \ldots, N+1$. Normalizing $ \mathbf{u}$ to components of unit length, we get the vector $\widetilde{\mathbf{u}} := \sqrt{N+1}{\mathbf{u}}\colon |\widetilde{{u}}_k|=1$, orthogonal to the rows of the matrix $A_{N+1}$. By virtue of the second equality in \eqref{eqv-syst0}, we have
    \[
    \sum_{i=1}^{N}R_{s_i}(\xi^k)\overline{R_{s_i}(\xi^l)} = - \widetilde{u}_k \overline{\widetilde{u}_l}.
    \]
 Since $R_{s_i}(\xi^k)\overline{R_{s_i}(\xi^l)} \in \{\pm 1, \pm i\}$, then $\widetilde{u}_k \overline{\widetilde{u}_l} = a+ib$, where $a, b \in \mathbb{Z}$. On the other hand $|\widetilde{u}_k \overline{\widetilde{u}_l}|=1$, which means $\widetilde{u}_k \overline{\widetilde{u}_l} \in \{\pm 1, \pm i\}$, i.e. $\widetilde{u}_k = \alpha_k \widetilde{u}_{N+1}$, where $\alpha_k \in \{\pm1,\pm i\}$, $k =1, \ldots N$ and $\widetilde{u}_{N+1}$ can be selected from $\{\pm1,\pm i\}$. We get that there is a matrix $ \widetilde{A} =\begin{pmatrix}
A_{N+1} \\
\widetilde{\mathbf{u}}
\end{pmatrix}$, which is a generalized Hadamard matrix over $\mathbb{Z}_4$.

Now we prove the statement of the Theorem. Let Seberry's conjecture be true, then if $N$ is even, then there exists a generalized Hadamard matrix over $ \mathbb{Z}_4$, which means that by Theorem \ref{theorem-1} $m_{+}(\mathbf{R}_{\mathbf{s}_N}^4) = m(\mathbf{R}_{\mathbf{s}_N}^4)=N$. If $N$ is odd, then under the conditions of Seberry's conjecture there exists a generalized Hadamard matrix over $ \mathbb{Z}_4$ of order $N+1$ and according to the statement proved at the beginning of the proof, $m_{+}(\mathbf{R}_{\mathbf{s}_N}^4) = m(\mathbf{R}_{\mathbf{s}_N}^4)=N+1$. Let conjecture \textbf{A$^\prime$} be true for subspaces $\mathbf{R}_{\mathbf{s}_N}^4$, then for all odd $N$ we have $m_{+}(\mathbf{R}_{\mathbf{s}_N}^4) = m (\mathbf{R}_{\mathbf{s}_N}^4)=N+1$, where the equality with $N+1$ due to Theorem \ref{theorem-1}. And according to the statement above, this implies the existence of generalized Hadamard matrices over $\mathbb{Z}_4$ for all even $N+1$.
\end{proof}

Next, we consider the case $p=3$ in Theorem \ref{theorem-4} (see also Remark \ref{why-yes} and Corollary \ref{cordop3}) and make Remark \ref{why-not} about why the argument of Theorem \ref{theorem-4} does not work for other $p$ and what can nevertheless be extracted from it in the case of prime $p$.

\begin{theorem} \label{theorem-4}
Let $p = 3$. Then for $N \equiv r \pmod p$, where $r = 1, \ldots, p-1$, the following estimate holds:
\begin{equation} \label{theorem-4-main}
       m_{+}(\mathbf{R}_{\mathbf{s}_N}^p) \ge N+p-r. 
\end{equation}
\end{theorem}
\begin{proof}
Let there be points $ \ xi^1, \ldots, \xi^m$ and positive weights $\lambda_1, \ldots \lambda_m$ such that $\mathbf{R}^p_{\mathbf{s}_N} \in \mathscr{M}_+^{\omega}(m, 2, 0)$. Then, according to Lemma \ref{lemma-baza}, the system \eqref{main} is executed. By grouping the same values of $R_{s_i}(\xi^k)\overline{R_{s_j}(\xi^k)}$, the sum in the second equality in \eqref{main} can be written as follows:
\begin{align*} 
    \sum_{k=1}^{m} \lambda_k R_{s_i}(\xi^k)\overline{R_{s_j}(\xi^k)} = \sum_{s=1}^{p-1}\sigma_s \omega_p^s, \ & \text{where $\sigma_s = \sum_{k} \lambda_k$ and summation by $k$:} \\
    & \text{$R_{s_i}(\xi^k)\overline{R_{s_j}(\xi^k)} = \omega^s$}. 
\end{align*} 
Then the  system  \eqref{main} in the new notation looks like this:
\begin{equation} \label{op2}
\begin{cases}
    \sum_{s=1}^{p-1}\sigma_s =1, \\
    \sum_{s=1}^{p-1}\sigma_s \omega_p^s = 0.
\end{cases}
\end{equation}
The minimal polynomial $ \omega_p$ for $p=3$ over $\mathbb{R}$ has the form
\begin{equation} \label{op3}
1+\omega_p + \ldots +\omega_p^{p-1} = 0. 
\end{equation}
If $\sigma_{p-1} = 0$, then the second equation in \eqref{op2} would specify a polynomial of degree less than the minimum polynomial and equals zero on $ \omega_p$, which is impossible. Since $\sigma_{p-1} \neq 0$, then dividing the second equation in \eqref{op2} by it, we obtain the minimal polynomial by virtue of the uniqueness of the minimal polynomial. Then from \eqref{op3} we get that $\sigma_{0} = \ldots = \sigma_{p-1}$, and from the first equality in \eqref{op2} we get
\begin{equation} \label{op4}
    \sigma_{0} = \ldots = \sigma_{p-1} = \frac{1}{p}. 
\end{equation}

On the other hand, since the weights are positive by Lemma \ref{lemma-dop}, we have:
\begin{equation} \label{op5}
0< \lambda_k \le \frac{1}{N}, \quad k=1, \ldots m.
\end{equation}
So from \eqref{op4} and \eqref{op5}, each $ \sigma_s$ is the sum of at least $\lceil \frac{N}{p} \rceil$ weights $\lambda_k$. Means,
\begin{equation}  \label{op6}
m \ge p\left\lceil \frac{N}{p} \right\rceil =  p\left\lceil \frac{qp+r}{p} \right\rceil = p(q+1) =N+p-r.
\end{equation}
Since the estimate \eqref{op6} is true for every set of positive weights, it is also true for the minimum one, i.e. for $p=3$ we have $m_{+}(\mathbf{R}_{\mathbf{s}_N}^p) \ge N+p-r$. 
\end{proof}
\begin{remark} \label{why-yes}
According to Corollary \ref{cordop}, $m(\mathbf{R}_{\mathbf{s}_N}^p)=N+1$ for $N \not\equiv 0 \pmod 3$. Let us analyze the statement of Theorem \ref{theorem-4} and compare the quantities $m(\mathbf{R}_{\mathbf{s}_N}^p)$ and $m_{+}(\mathbf{R}_{\mathbf{s}_N}^p)$ for $N \not\equiv 0 \pmod 3$. In the case $N \equiv 1 \pmod 3$ we obtain $m_{+}(\mathbf{R}_{\mathbf{s}_N}^3) \ge N+2 > N+1= m(\mathbf{R}_{\mathbf{s}_N}^p)$. In the case $N \equiv 2 \pmod 3$ we obtain $m_{+}(\mathbf{R}_{\mathbf{s}_N}^3) \ge N+1= m(\mathbf{R}_{\mathbf{s}_N}^p)$; if, in addition, it is known that there exists a generalized Hadamard matrix over $\mathbb{Z}_3$ of order $N+1 \equiv 0 \pmod 3$, then by Theorem \ref{theorem-1} $m_{+}(\mathbf{R}_{\mathbf{s}_N}^3) = N+1$. However, examples are known in which, for $N \equiv 0 \pmod 3$, no generalized Hadamard matrix over $\mathbb{Z}_3$ exists.
\end{remark}
\begin{corollary} \label{cordop3}
    Let $p=3$. Then for $N\equiv 1\pmod3$ the strict inequality holds: $m_+(\mathbf{R}_{\mathbf{s}_N}^3)>m(\mathbf{R}_{\mathbf{s}_N}^3)=N+1$.
\end{corollary}
\begin{remark} \label{why-not}
    For $p=2$, the argument of Theorem \ref{theorem-4} yields only the trivial estimate $m_{+}(\mathbf{R}_{\mathbf{s}_N}^p) \ge m(\mathbf{R}_{\mathbf{s}_N}^p)$. The argument of Theorem \ref{theorem-4} is not valid for other $p > 3$, since the degree of the minimal polynomial of the element $\omega_p$ over $\mathbb{R}$ is 2 and it has a form different from \eqref{op3}. However, if one considers $p$ prime and coefficients from $\mathbb{Q}$ in a representation of the form \eqref{eq:disc}, then the minimal polynomial of the element $\omega_p$ over $\mathbb{Q}$ has a form \eqref{op3}. Consequently, for $N\equiv r\pmod p$, where $r=1, \ldots, p-1$, the following estimate holds: $m_{+}^{\mathbb{Q}}(\mathbf{R}_{\mathbf{s}_N}^p) \ge N+p-r$. Nevertheless, this does not contradict the existence of a representation of the form \eqref{eq:disc} with irrational positive weights (see Remark \ref{hard}, Example~2).
\end{remark}

In Theorem \ref{theorem-3}, using the necessary condition for the existence of classical Hadamard matrices on the divisibility of the matrix order by 4, we show that even for classical Rademacher functions for infinity many  $N$, an exact discretization of the $L_2$-norm for a minimum number of nodes is impossible with positive weights. 

\begin{theorem} \label{theorem-3}
Let $p=2$ and $N\ge 5$. Then for $N\equiv r\pmod4$, where $r=1,2$, the strict inequality holds: $m_+(\mathbf{R}_{\mathbf{s}_N}^2)>m(\mathbf{R}_{\mathbf{s}_N}^2)=N+1$. For $N \equiv 3\pmod4$, under the assumption that the Hadamard conjecture is true (see Remark \ref{about-matrix-2}), the equality holds: $m_+(\mathbf{R}_{\mathbf{s}_N}^2)=m(\mathbf{R}_{\mathbf{s}_N}^2)=N+1$. Conversely, if one assumes that for all $N \equiv 3 \pmod 4$ the equality $m_+(\mathbf{R}_{\mathbf{s}_N}^2)=m(\mathbf{R}_{\mathbf{s}_N}^2)$ holds, i.e., that conjecture $\mathbf{A^\prime}$ is true for the spaces $\mathbf{R}_{\mathbf{s}_N}^2$ for such $N$, then the Hadamard conjecture is also true.
\end{theorem}
\begin{proof}
For $N \equiv r \pmod4,\quad N>2, \quad r = 1, 2$ by Corollary \ref{cordop2} $m(\mathbf{R}_{\mathbf{s}_N}^2)=N+1.$ It remains to prove that for such $N$ points, one cannot choose $N+1$ points such that all the weights are positive.

Assume the opposite. Let there be points $\xi^1, \ldots, \xi^{N+1}$ and positive weights $\lambda_1, \ldots \lambda_{N+1}$ such that $\mathbf{R}^p_{\mathbf{s}_N} \in \mathscr{M}_+^{\omega}(N+1, 2, 0)$. Then, according to Lemma \ref{lemma-dop2}, the system \eqref{eqv-syst0} is executed: 
\begin{equation} \label{eqv-syst}
    \begin{cases}
1 = N \lambda_k + u_k^2, \quad k =1, \ldots, N+1 \\[6pt]
0 = \sqrt{\lambda_k \lambda_l} \sum_{i=1}^{N}r_{s_i}(\xi^k)r_{s_i}(\xi^l) + u_k u_l, \quad k \neq l.
\end{cases}
\end{equation}

Let $ N \equiv1\pmod4$. In the case of odd $N$, according to Lemma \ref{lemma-dop3}, we have
\begin{equation} \label{949}
    \lambda_k  =  \frac{1}{N+1},  \quad u_k = \frac{1}{N+1}, \qquad \forall k \in \{1, \ldots, N+1\}.
\end{equation}

Let $\widetilde{u}_k:= \sqrt{N+1} u_k$. Then, by virtue of \eqref{949}, $\widetilde{u}_k \in \{\pm1\}$. Taking \eqref{Ort} into account,
\begin{equation} \label{ADAMAR}
  E_{N+1} = \widetilde{B}^{T} \widetilde{B} = \frac{1}{N+1}\widetilde{A}^{T} \widetilde{A}, \quad \text{где} \; \widetilde{A} =\begin{pmatrix}
A_{N+1} \\
\widetilde{\mathbf{u}}
\end{pmatrix}.  
\end{equation}

That is, $ \widetilde{A}$ is a Hadamard matrix, which means $4 \mid (N+1)$, but this cannot happen because $N\equiv1\pmod4$. We get a contradiction, which means that $N+1$ cannot have positive weights. The case $N \equiv1\pmod4$ is proved. 

Let $ N \equiv2\pmod4$. For even $N$, the sum $\sum_{i=1}^{N}r_{s_i}(\xi^k)r_{s_i}(\xi^l)$ is also even. If $ \sum_{i=1}^{N}r_{s_i}(\xi^k)r_{s_i}(\xi^l) = 0$, then from the second equation in \eqref{eqv-syst} we have either $u_k = 0$ or $u_l = 0$. If $u_k \neq 0$ and $u_l \neq 0$, then $ |\sum_{i=1}^{N}r_{s_i}(\xi^k)r_{s_i}(\xi^l)| \ge 2$ and also from the second equation in \eqref{eqv-syst} we have $u_k^2 u_l^2 \ge 4\lambda_k \lambda_l$. Given the first equation in \eqref{eqv-syst}, we get $(1-N\lambda_k) (1-N\lambda_l) \ge 4 \lambda_k \lambda_l$ and simplifying we come to
\begin{equation} \label{eqv-3}
 1- N \lambda_k - N  \lambda_l + (N^2 -4)\lambda_k \lambda_l \ge 0.
\end{equation}
Let $r$ denote the number of $u_k$ not equal to 0, and $ S= \sum\limits_{k\colon u_k \neq 0}\lambda_k$. Summing \eqref{eqv-3} over the pairs $(k, l)$ such that $k \neq l$, $u_k \neq 0$, and $u_l \neq 0$, we obtain

\begin{equation} \label{eqv-4}
    r(r-1) - 2N(r-1)S + (N^2 - 4)\left(S^2 - \sum\limits_{k\colon u_k \neq 0}\lambda_k^2\right) \geq 0.
\end{equation}

Using the Cauchy-Bunyakovsky inequality,
\[
S^2 - \sum\limits_{k\colon u_k \neq 0} \lambda_k^2 \leq S^2 - \frac{S^2}{r} = S^2 \cdot \frac{r-1}{r}.
\]

We substitute this  into the inequality  \eqref{eqv-4}:
\[
r(r-1) - 2N(r-1)S + (N^2 - 4)S^2 \cdot \frac{r-1}{r} \geq 0.
\]
Then we get the system

\begin{equation} \label{eqv-syst-2}
    \begin{cases}
r^2 - 2NrS + (N^2 - 4)S^2 \geq 0 \\[6pt]
S \le \frac{r}{N},
\end{cases}
\end{equation}
where the second equality follows from the definition of $S$ and Lemma \ref{lemma-dop}.
The solution to \eqref{eqv-syst-2} is $S \le \frac{r}{N+2}$ (the roots of the quadratic equation in \eqref{eqv-syst-2} $S=\frac{r}{N \pm2}$).
Then from the estimate for $S$ and the equality $ \lambda_k = \frac{1}{N}$ for $u_k = 0$, we get
\begin{equation} \label{eqv-5}
1 = \sum_{k=1}^N \lambda_k = S + \frac{N+1-r}{N}\le \frac{r}{N+2} + \frac{N+1-r}{N}.
\end{equation}
From \eqref{eqv-5}, we get that $r \le \frac{N}{2} +1$. This means that the number of non-zero $u_k$ among $u_1, \ldots, u_{N+1}$ is greater than or equal to $\frac{N}{2}$. Then, since $N \ge  6$, there are always at least three $u_{k_1} = u_{k_2} = u_{k_3} =0$. In this case, from the second equation in \eqref{eqv-syst-2} $\sum_{i=0}^{N-1} r_i(\xi^{k_p})r_i(\xi^{k_q}) =0 $ for $p, q =1, 2, 3$ and $p \neq q$. Then
\begin{equation} \label{w}
    \sum_{i=0}^{N-1} \sum_{p,q=1,\atop p\neq q}^{3} r_i(\xi^{k_p})r_i(\xi^{k_q})  =0.
\end{equation}
Calculate the internal sum in \eqref{w}:
\begin{equation} \label{ww}
\sum_{p,q=1,\atop p\neq q}^{3} r_i(\xi^{k_p})r_i(\xi^{k_q}) =     \begin{cases}
3, \quad \text{if $r_i(\xi^{k_1})=r_i(\xi^{k_2}) = r_i(\xi^{k_3})$}, \\[6pt]
-1, \quad \text{otherwise}.
\end{cases}
\end{equation}
Let the indices $i$ of the first type from \eqref{ww} be $ M$, then from \eqref{w} we get
\begin{equation} \label{www}
3M - (N-M) =0.
\end{equation}
From \eqref{www} we get that $4 \mid N$, which contradicts that $N \equiv2\pmod4$. Thus, this case is also proved.

Let $N \equiv3\pmod4$ and assume that the Hadamard conjecture holds, then there exists a Hadamard matrix of order $N+1$ and by Theorem \ref{theorem-1} $m_{+}(\mathbf{R}_{\mathbf{s}_N}^2) = N+1= m(\mathbf{R}_{\mathbf{s}_N}^2)$. Conversely, if we assume that for all $N \equiv 3 \pmod 4$ the equality $m_+(\mathbf{R}_{\mathbf{s}_N}^2)= m (\mathbf{R}_{\mathbf{s}_N}^2)$ holds, then we can construct a Hadamard matrix $\widetilde{A}$ of order $N+1$ (see \eqref{ADAMAR}) in the same way as the proof for the case $N\equiv1 \pmod 4$. That is, if $m_+(\mathbf{R}_{\mathbf{s}_N}^2)=m (\mathbf{R}_{\mathbf{s}_N}^2)$ for all $N \equiv3 \pmod4$, then the Hadamard matrix $\widetilde{A}$ exists for all $N+1 \equiv 0 \pmod4$.
\end{proof}
\begin{corollary}
 If the Hadamard conjecture is correct, then
\[
m_+(\mathbf{R}_{\mathbf{s}_N}^2)-m(\mathbf{R}_{\mathbf{s}_N}^2)\le 2.
\]
\end{corollary}

\end{document}